\documentclass[11pt,psamsfonts]{amsart}

\newtheorem{theorem}{Theorem}[section]
\newtheorem{lemma}[theorem]{Lemma}

\newtheorem{remark}[theorem]{Remark}

\theoremstyle{definition}

\theoremstyle{remark}

\newcommand{\R}{{\mathbb R}}

\newcommand{\ve}{{\varepsilon}}
\newcommand{\vp}{{\varphi}}
\newcommand{\p}{{\partial}}

\newcommand{\mc}{\mathcal }
\newcommand{\mf}{\mathbf }
\newcommand{\mb}{\mathbb }
\newcommand{\wt}{\widetilde }

\numberwithin{equation}{section}

\usepackage{graphicx}
\usepackage{amssymb}
\usepackage{amscd}
\usepackage{amsmath}
\usepackage{amsfonts}
\usepackage{amsbsy}
\usepackage{epsfig}

\begin{document}

\title[Inverse Jacobian multipliers and Hopf bifurcation]{Inverse Jacobian multipliers and Hopf bifurcation on center manifolds}

% author one information
%\author[J. Llibre]{Jaume Llibre}
%\address{Departament de Matem\`{a}tiques, Universitat Aut\`{o}noma de Barcelona, 08193 Bellaterra, Barcelona, Catalonia, Spain}
%\curraddr{}
%\email{jllibre@mat.uab.cat}
%\thanks{}

% author two information
%\author[K. Wu]{Kesheng Wu}
%\address{Department of Mathematics, Shanghai Jiao Tong University,  Shanghai 200240,  People's Republic of China.}
%\curraddr{}
%\email{E-mail: kshengwu@gmail.com}
%\thanks{The author is partially supported by  FP7-PEOPLE-2012-IRSES-316338 of Europe.}

%    author three information
\author[X. Zhang]{Xiang Zhang}
\address{Department of Mathematics, and MOE-LSC, Shanghai Jiao Tong University,  Shanghai 200240,  People's Republic of China.}
%\curraddr{}
\email{xzhang@sjtu.edu.cn}
%\thanks{}

\subjclass[2010]{34A34, 37C10, 34C14, 37G05. }

\date{}

\dedicatory{}

\keywords{Inverse Jacobian multiplier; Hopf bifurcation; normal form; center; focus.}

\begin{abstract}
In this paper we consider a class of higher dimensional differential systems in $\mathbb R^n$ which have a two dimensional center
manifold at the origin with a pair of pure imaginary eigenvalues. First we characterize the existence of either analytic or $C^\infty$ inverse
Jacobian multipliers of the systems around the origin, which is either a center or a focus on the center manifold. Later we
study the cyclicity of the system at the origin through Hopf bifurcation by using the vanishing multiplicity of the inverse Jacobian multiplier.
\end{abstract}

\maketitle

\bigskip

\section{Background and statement of the main results}\label{s1}

For real planar differential systems, the problems on center--focus and Hopf bifurcation are classical and related.
They are important subjects in the bifurcation theory and also in the study of the Hilbert's 16th problem \cite{CL07, DLA06,RS09,Ye86}.

For planar non--degenerate center,
Poincar\'e provided an equivalent characterization.

\noindent{\bf Poincar\'e center Theorem}. {\it For a real planar analytic differential system with the origin as a singularity having a pair of pure imaginary eigenvalues, then the origin is a center if and only if the system has a local analytic first integral, and if and only if the system is analytically equivalent to
\[
\dot u=-\mathbf i u(1+g(uv)),\quad \dot v=\mathbf i v(1+g(uv)),
\]
with $g(uv)$  without constant terms, where we have used the conjugate complex coordinates instead of the two real ones.
}

This result has a higher dimensional version, see for instance \cite{LPW12,Zh08,Zh13}, which characterizes the equivalence between the analytic integrability and
the existence of analytic normalization of analytic differential systems to its Poincar\'e--Dulac normal form of a special type.

Reeb \cite{Re52} in 1952 provided another characterization on planar centers via inverse integrating factor. Recall that a function $V$ is an {\it inverse integrating factor} of a planar differential system if  $1/V$ is an integrating factor of the system. From \cite{EP09, GGG10,GG10,GLV96} we know that inverse integrating factors have better properties than integrating factors.

\noindent{\bf Reeb center Theorem}. {\it Real planar analytic differential system
 \[
 \dot x=-y+f_1(x,y),\qquad \dot y=x+f_2(x,y),
 \]
has the origin as a center if and only if it admits a real analytic local inverse integrating factor with non--vanishing constant part.
}

Poincar\'e center theorem was extended to higher dimensional differential systems which have a two dimensional center manifold by Lyapunov.
Consider analytic differential systems in $\mathbb R^n$
\begin{eqnarray}\label{e1}
\dot x&=&-y+f_1(x,y, z)=F_1(x,y,z),\nonumber\\
\dot y&=&\,\,\,\,\, x+f_2(x,y, z)=F_2(x,y,z),\\
\dot z&=& Az+f(x,y,z)=F(x,y,z),\nonumber
\end{eqnarray}
with $z=(z_3,\ldots,z_n)^{tr}$, $A$ is a real square matrix of order $n-2$, and $f=(f_3,\ldots,f_n)^{tr}$ and $F=(F_3,\ldots,F_n)^{tr}$.
Hereafter we use $ tr$ to denote the transpose of a matrix.
Moreover we assume that $\mf f:=(f_1,f_2,f)=O(|(x,y,z)|^2)$ are $n$ dimensional vector valued analytic  functions. We denote by
\[
\mathcal X=F_1(x,y,z)\frac{\partial }{\partial x}+F_2(x,y,z)\frac{\partial }{\partial y}+\sum\limits_{j=3}\limits^n F_j(x,y,z)\frac{\partial }{\partial z_j}
\]
the vector field associated to systems \eqref{e1}.

Assume that the eigenvalues of $A$ all have non--zero real parts. Then from the Center Manifold Theorem we get that system \eqref{e1} has
a center manifold tangent to the $(x,y)$ plane at the origin (of course center manifolds are not necessary unique, and may not be analytic even not $C^\infty$). Moreover this center manifold can be represented as
\begin{equation}\label{ec}
\mathcal M^c=\bigcap_{j=3}^n\{z_j=h_j(x,y)\}.
\end{equation}

\noindent{\bf Lyapunov center Theorem}. {\it Assume that $A$ has no eigenvalues with vanishing real parts. The following statements hold.
\begin{itemize}
\item[$(a)$] System \eqref{e1} restricted to the center manifold has the origin as a center if and
only if it admits a real analytic local first integral of the form $\Phi(x,y,z) = x^2 + y^2 +$\,{\rm higher order term}  in a neighborhood of the origin in $\mathbb R^n$.
\item[$(b)$] If the condition in statement $(a)$ holds, then the center manifold is unique and analytic.
\end{itemize}}

For a proof of the Lyapunov center Theorem, we refer to \cite{Si85} and \cite[Theorems 3.1, 3.2 and $\S$5]{Bi79}.

Reeb center Theorem via inverse integrating factor was extended to differential systems in $\mathbb R^3$ by Buic\u{a}, Garc\'ia and Maza \cite{BGM12}. A smooth function $J(x)$ is an {\it inverse Jacobian multiplier} of system \eqref{e1} if
\[
\mathcal X(J)=J\mbox{div}\mathcal X.
\]
In fact, if $J(x)$ is an inverse Jacobian multiplier of system \eqref{e1} then $1/J$ is a Jacobian multiplier of the system, i.e.
\[
\partial_{x}\left(\frac{F_1}{J}\right)+\partial_{y}\left(\frac{F_2}{J}\right)+\partial_{z_3}\left(\frac{F_3}{J}\right)
   +\ldots+\partial_{z_n}\left(\frac{F_n}{J}\right)=0,
\]
where $\partial_x$ denotes the partial derivative with respect to $x$.

Buic\u{a} {\it et al}'s main results in \cite{BGM12} can be summarized as following.

\noindent{\bf Buic\u{a}, Garc\'ia and Maza center--focus Theorem}. {\it Assume that system \eqref{e1} is defined in $\mathbb R^3$ and $A$ is a
non--zero real number. The following statements hold.
\begin{itemize}
\item[$(a)$] System \eqref{e1} restricted to the center manifold has the origin as a center if and
only if it admits an analytic local inverse Jacobian multiplier of the form $J(x,y,z) = z+$\,{\rm higher order term}  in a neighborhood of the origin in $\mathbb R^3$. Moreover, if such an inverse Jacobian multiplier exists, then the analytic center manifold $\mathcal M^c\subset J^{-1}(0)$.
\item[$(b)$] If system \eqref{e1} restricted to the center manifold has the origin as a focus, then there exists a local $C^\infty$ and
non--flat inverse Jacobian multiplier of the form $J(x,y,z) = z(x^2+y^2)^k+$\,{\rm higher order term} with $k\ge 2$, in a neighborhood of the origin in $\mathbb R^3$. Moreover, there exists a local $C^\infty$ center manifold $\mathcal M$ such that  $\mathcal M\subset J^{-1}(0)$.
 \end{itemize}}

In this paper we will extend Buic\u{a} {\it et al}'s results to any finite dimensional differential system \eqref{e1}. We should say that this extension is not
trivial, because for higher dimensional differential systems we need new ideas and techniques than those in \cite{BGM12,BGM13}. Parts of the methods in
\cite{BGM12,BGM13} are only suitable for
three dimensional differential systems but not for higher dimensional ones.

Let $\lambda_3,\ldots,\lambda_n$ be the eigenvalues of the matrix $A$. Then system \eqref{e1} at the origin has the eigenvalues $\lambda=(\mathbf i,-\mathbf i,\lambda_3,\ldots,\lambda_n)$, where $\mathbf i=\sqrt{-1}$. Let
\[
\mathcal R=\left\{k\in\mathbb Z^n:\,\, \langle k,\lambda\rangle=0, \,k+e_j\in\mathbb Z_+^n,\, j=3,\ldots,n\right\},
\]
where $\mathbb Z_+$ denotes the set of non--negative integers, $e_j$ is the unit vector with its $j$th component equal to $1$ and the others all vanishing,
and $\langle k,\lambda\rangle=k_1\mathbf i-k_2\mathbf i+\sum\limits_{j=3}\limits^nk_j\lambda_j$. We remark that in the definition $\mathcal R$ we choose $k\in\mathbb Z^n$ but not $k\in\mathbb Z_+^n$, because we will also discuss the case $\langle k,\lambda\rangle=\lambda_j$ for $k\in\mathbb Z_+^n$ and $j\in\{1,\ldots,n\}$. 

In this paper we have a basic assumption.
\[
(H) \quad \mathcal R \mbox{ is one dimensional and }  A \mbox { can be diagonalizable in $\mb C$}.\qquad\qquad
\]

Clearly if $A$ has its eigenvalues either all having positive real parts or all having negative real parts, then
$\mc R$ has only one linearly  independent element with generator $(1,1,0)$. For three dimensional differential systems of the form \eqref{e1}, this condition
always holds provided that $A$ is a nonzero real number.

By the assumption $(H)$ we get easily that $\mbox{Re}\, \lambda_j\ne 0$ for $j=3,\ldots,n$. So from the Center Manifold Theorem we get that
system \eqref{e1} has a center manifold tangent to the $(x,y)$ plane at the origin, and it can be represented in \eqref{ec}.

In the case that $A$ has complex eigenvalues, we assume without loss of generality that there exists an $m\in \mathbb Z_+$ with $2m\le n-2$ such that $\lambda_{3+2j}$ and $\lambda_{3+2j+1}$, $j=0,1,\ldots,m-1$, are
conjugate complex eigenvalues of $A$. Of course if $m=0$ then all the eigenvalues are real.

Our first result provides an equivalent characterization on the center on the center manifold $\mathcal M^c$  at the origin via inverse Jacobian multipliers.
\begin{theorem}\label{t1}
Assume that the analytic differential system \eqref{e1} satisfies $(H)$ and the eigenvalues of $A$ either all having positive real parts or all having negative real parts. The following statements hold.
\begin{itemize}
\item[$(a)$] System \eqref{e1} restricted to $\mathcal M^c$ has the origin as a center if and only if the system has a local analytic inverse Jacobian multiplier of the form
    \begin{eqnarray}\label{*1}
    J(x,y,z)&=&\prod\limits_{j=0}\limits^{m-1}\left[(z_{3+2j}-p_{3+2j}(x,y,z))^2+(z_{3+2j+1}-p_{3+2j+1}(x,y,z))^2\right]\nonumber\\
     && \qquad \times\prod\limits_{l=3+2m}\limits^n (z_l-p_l(x,y,z))V(x,y,z),
    \end{eqnarray}
     in a neighborhood of the origin in $\mathbb R^n$, where $p_j=O(|(x,y,z)|^2)$ for $j=3,\ldots,n$, and $V(0,0,0)=1$. For $m=0$ the first product does not appear.
\item[$(b)$] If system \eqref{e1} has the inverse Jacobian multiplier as in statement $(a)$, then the center manifold $\mc M^c$
 is unique and analytic, and $\mathcal M^c\subset J^{-1}(0)$.
\end{itemize}
\end{theorem}

We note that the set of matrices satisfying $(H)$ is a full Lebesgue measure subset in the set of real matrices of order $n$.

The second result shows the existence of $C^\infty$ smooth local inverse Jacobian multiplier provided that the origin on the center manifold is a focus.

\begin{theorem}\label{t2}
Assume that the differential system \eqref{e1} satisfies $(H)$. The following statements hold.
\begin{itemize}
\item[$(a)$] If system \eqref{e1} restricted to $\mathcal M^c$ has the origin as a focus, then the system has a local $C^\infty$
inverse Jacobian multiplier of the form
\begin{eqnarray}\label{*2}
    J(x,y,z)&=&\prod\limits_{j=0}\limits^{m-1}\left[(z_{3+2j}-p_{3+2j}(x,y,z))^2+(z_{3+2j+1}-p_{3+2j+1}(x,y,z))^2\right]\nonumber\\
     &&  \times\prod\limits_{s=3+2m}\limits^n (z_s-p_s(x,y,z))\left[(x-q_1(x,y,z))^2+(y-q_2(x,y,z))^2\right]^l\\
 && \times h\left((x-q_1(x,y,z))^2+(y-q_2(x,y,z))^2\right)V(x,y,z),\nonumber
    \end{eqnarray}
    in a neighborhood of the origin in $\mathbb R^n$, where $l\ge 2$,  $p_j, q_i=O(|(x,y,z)|^2)$, and $h(0)=V(0,0,0)=1$.

\item[$(b)$] There exists a local $C^\infty$ center manifold $\mathcal M$ such that $\mathcal M\subset J^{-1}(0)$.
\end{itemize}
\end{theorem}
We call $l$ {\it vanishing multiplicity} of the inverse Jacobian multiplier.

Next we will study the Hopf bifurcation of system \eqref{e1} under small perturbations through inverse Jacobian multipliers. In this direction the first study
is due to Buic\u{a}, Garc\'{\i}a and Maza \cite{BGM13} for a three dimensional differential system.

Consider an analytic perturbation of system \eqref{e1} in the following form
\begin{eqnarray}\label{e5}
\dot x&=&-y+g_1(x,y, z,\ve)=G_1(x,y,z,\ve),\nonumber\\
\dot y&=&\,\,\,\,\, x+g_2(x,y, z,\ve)=G_2(x,y,z,\ve),\\
\dot z&=& Az+g(x,y,z,\ve)=G(x,y,z,\ve),\nonumber
\end{eqnarray}
where $\ve\in \R^m$ is an $m$ dimensional parameter and $\|\ve\|\ll 1$, $\mf g:=(g_1,g_2,g)=O(|(x,y,z)|)$ are analytic in a neighborhood of the origin,
and $\mf g(x,y,z,0)=\mf f(x,y,z)$ with $\mf f$ defined in \eqref{e1}. These conditions make sure that the origin is always a singularity of system
\eqref{e5} for all $\|\ve\|\ll 1$. In addition, in order to keep the monotone property of the origin, we assume that the determinant of the Jacobian matrix of
$\mf G=(G_1,G_2,G)$ with  respect to $(x,y,z)$ at the origin has the eigenvalues
\[
\alpha(\ve)\pm \mf i,\quad \lambda_j+\mu_j(\ve),\qquad j=3,\ldots,n,
\]
satisfying $\alpha(0)=\mu_j(0)=0$. For convenience we denote by $\mc X_\ve$ the vector field associated to \eqref{e5}. Then $\mc X_0=\mc X$.

Next we shall study the Hopf bifurcation of system \eqref{e5} at the origin when the parameters $\ve$ vary near $0\in \mathbb R^m$.
That is, when the values of $\ve$ change, the stability of the origin of system \eqref{e5}  will probably change, and so there bring appearance or disappearance
of small amplitude limit cycles of system \eqref{e5} which are bifurcated from the origin, i.e. if $\ve$ tend to $0$ these limit cycles
will approach to the origin. The maximal number of limit cycles which can be bifurcated from the Hopf at the origin of systems \eqref{e5} is called {\it cyclicity} of  system \eqref{e1} at the origin under the perturbation \eqref{e5}. Denote this number by Cycl$(\mc X_\ve,0)$.

Now we can state our third result on the Hopf bifurcation.

\begin{theorem}\label{t3}
Assume that the analytic differential system \eqref{e1} satisfies $(H)$. If system \eqref{e1} restricted to $\mathcal M^c$ has the origin as a focus,
then \mbox{\rm Cycl}$(\mc X_\ve,0)=l-1$, where $l$ is the vanishing multiplicity of the inverse Jacobian multiplier defined in Theorem \ref{t2}.
\end{theorem}

This result is an extension of the main result of \cite{BGM13} to any finite dimensional differential systems.

For the real differential system \eqref{e1} there always exists an invertible linear transformation which sends $A$ to its Jordan normal form.
So in what follows we assume without loss of generality that $A$ in system \eqref{e1} is in the real Jordan normal form.

In the rest of this paper we will prove our main results. In the next section we will prove Theorems \ref{t1} and \ref{t2}. The proof of Theorem \ref{t3} will
be given in Section \ref{s3}.

\medskip
\section{Proof of Theorems \ref{t1} and \ref{t2}}\label{s2}

\setcounter{section}{2}
\setcounter{equation}{0}\setcounter{theorem}{0}

\subsection{Preparation to the proof}

For simplifying notations we will use conjugate complex coordinates instead of the real ones which correspond to conjugate complex eigenvalues of
the linear part of system \eqref{e1} at the origin.

Set $\xi=x+\mathbf i y,\, \eta=x-\mathbf iy$. Since $A$ is real, if it has complex eigenvalues, they should appear in pair. Corresponding
to a pair of conjugate complex eigenvalues of $A$, the associated coordinates are $z_j$ and $z_{j+1}$ by assumption. Then instead of this pair of real coordinates
we choose a pair of conjugate complex coordinates $\zeta_j=z_j+\mathbf i z_{j+1}$ and $\zeta_{j+1}=z_j-\mathbf i z_{j+1}$.
Under these new coordinates system \eqref{e1} can be written in
\begin{eqnarray}\label{e1-1}
\dot \xi &=& - \mathbf i \xi+\widetilde f_1(\xi,\eta,\zeta)=\widetilde F_1(\xi,\eta,\zeta),\nonumber\\
\dot \eta &=&\,\,\,\,\,\mathbf i \eta +\widetilde  f_2(\xi,\eta,\zeta)=\widetilde F_2(\xi,\eta,\zeta),\\
\dot \zeta &=& B \zeta+ \widetilde f(\xi,\eta,\zeta)=\widetilde F(\xi,\eta,\zeta),\nonumber
\end{eqnarray}
with $B=\mbox{diag}(\lambda_3,\ldots,\lambda_n)$, where we have used the assumption $(H)$ and the fact that $A$ is in the real Jordan normal form.
Denote by $\widetilde {\mathcal X}$ the vector field associated to system \eqref{e1-1}. We note that system \eqref{e1-1} is different from system \eqref{e1} only
in a rotation. But using the coordinates $(\xi,\eta,\zeta)$, some expressions will be simpler than in the coordinates $(x,y,z)$. This idea was first
introduced in \cite{Zh11}.

First we recall a basic fact on inverse Jacobian multipliers of vector fields under transformations, which will be used in the full paper.
\begin{lemma}\label{l0}
Let $\mathcal X$ be the vector field associated to system \eqref{e1} and $J$ be an inverse Jacobian multiplier of $\mathcal X$.
Under an invertible smooth transformation of coordinates $(x,y,z)=\Phi(u,v,w)$, the vector field $\mathcal X$ becomes
\[
\dot {\mathbf w}=(D\Phi(\mathbf w))^{-1}\mathbf F\circ\Phi(\mathbf w),
\]
where $\mathbf F=(F_1,F_2,F)^{tr}$ and $\mf w=(u,v,w)^{tr}$. Then this last system has an inverse Jacobian multiplier
$\widetilde J(\mathbf w)=\frac{J(\Phi(\mathbf w))}{D\Phi(\mathbf w)}$.
\end{lemma}

Recall that hereafter we use $D\Phi$ to denote the determinant of the Jacobian matrix of $\Phi$ with respect to its variables.

In the proof of our main results we need the Poincar\'e--Dulac normal form theorem. For an analytic or formal differential system in $\mathbb R^n$
or $\mathbb C^n$
\begin{equation}\label{nfe1}
\dot x=Cx+f(x),
\end{equation}
with $C$ in the Jordan normal form, and $f(x)$ has no constant and linear part, the Poincar\'e--Dulac normal formal theorem shows that system \eqref{nfe1} can always be transformed to a system of the form
\[
\dot y=Cy+g(y),
\]
through a near identity transformation $x=y+\psi(y)$ with $\psi(0)=0$ and $\partial\psi(0)=0$, where $g(y)$ contains resonant terms only,
and $\partial\psi(y)$ denotes the Jacobian matrix of $\psi$ with respect to $y$. Recall that a monomial $y^ke_j$ in the $j$th component of $g(y)$
is {\it resonant} if $\mu_j=\langle k,\mu\rangle$, where $\mu=(\mu_1,\ldots,\mu_n)$ are the eigenvalues of $C$. The transformation from \eqref{nfe1}
to its normal form is called {\it normalization}.  Usually the normalization is not unique. If a normalization contains only non--resonant terms, then
it is called {\it distinguished normalization}. Distinguished normalization is unique. A monomial $x^k$ in a normalization or in a function is {\it resonant}
if $\langle k, \mu\rangle=0$.

In our case, by the Poincar\'e--Dulac normal form theorem we have the following result.

\begin{lemma}\label{l1}
Under the assumption $(H)$ system \eqref{e1-1} is formally equivalent to
\begin{eqnarray}\label{e2}
\dot u&=&- u(\mathbf i+g_1(uv)),\nonumber\\
\dot v&=&\,\,\,\,\,  v(\mathbf i+g_2(uv)),\\
\dot w_j&=&\, \,  w_j(\lambda_j+g_j(uv)) ,\qquad j=3,\ldots,n,\nonumber
\end{eqnarray}
through a distinguished normalization of the form
$(x,y,z)=(u,v,w)+ \ldots$,
where dots denote the higher order terms.
\end{lemma}

About the smoothness of the transformation in Lemma \ref{l1} we have the following results.

\begin{lemma}\label{l2}
Under the assumption $(H)$, for system \eqref{e1-1} to its Poincar\'e-Dulac normal form \eqref{e2}
the following statements hold.
\begin{itemize}
\item[$(a)$] If system \eqref{e1-1} restricted to the center manifold $\mathcal M^c$ has the origin as a focus, then the distinguished normalization
 is $C^\infty$.
\item[$(b)$] If system \eqref{e1-1} restricted to $\mathcal M^c$ has the origin as a center, and the eigenvalues of $A$ have either all positive real parts or all negative real parts, then the distinguished normalization is analytic.
\end{itemize}
\end{lemma}

\begin{proof} $(a)$  We note that $u$ and $v$ are conjugate in \eqref{e2}, we have $g_2=\overline g_1$. Since the origin of system \eqref{e2} on $w=0$ is a focus, it
follows that $\mbox{Re}\, g_1\ne 0$. So our vector fields \eqref{e2} are outside the exception set which was defined on page 254 of \cite{Be86}.
Hence we get from Theorem 1 of Belitskii \cite{Be86} that the distinguished normalization from systems \eqref{e1-1} to \eqref{e2} is
$C^\infty$.

\noindent $(b)$ Since the eigenvalues of $A$ have non--vanishing real parts, we have
\[
\lambda_j\ne k_1 (-\mathbf i)+k_2 \mathbf i=(k_2-k_1)\mathbf i,   \quad k_1,\,k_2\in\mathbb Z_+\quad \mbox{ for } j=3,\ldots,n.
\]
So by Theorem 10.1 of \cite{Bi79}, system \eqref{e1-1} is formally equivalent to
\begin{eqnarray}\label{e2-1}
\dot u&=&- u(\mathbf i+g_1(uv)),\nonumber\\
\dot v&=&\,\,\,\,\,  v(\mathbf i+g_2(uv)),\\
\dot \rho_j&=&\, \, \lambda_j \rho_j+h_j(u,v,\rho) ,\qquad j=3,\ldots,n,\nonumber
\end{eqnarray}
with $g_1,g_2=o(1)$, $h_j=O(|(u,v,\rho)|^2)$ and $h_j(u,v,0)=0$ for $j=3,\ldots,n$,  through a distinguished normalization of the form
\[
\xi=u+\psi_1(u,v,\rho),\quad \eta=v+\psi_2(u,v,\rho),\quad \zeta=\rho+\psi(u,v),
\]
where $\rho=(\rho_3,\ldots,\rho_n)$ and $\psi=(\psi_3,\ldots,\psi_n)$ with $\psi_1,\psi_2,\psi=O(|(u,v,\rho)|^2)$.
System \eqref{e2-1} is called a {\it quasi--normal form} of system \eqref{e1-1}, see \cite{Bi79}.

By the assumption system \eqref{e2-1}  has the origin as a center on the center manifold $w=0$ and so has a formal first integral.
By Zhang \cite{Zh08} we get that $g_1(uv)=g_2(uv)$ in \eqref{e2-1}.
Applying Theorems 10.2, 3.2 and $\S 5$ of \cite{Bi79} to our case, we get that the distinguished normalization from system \eqref{e1-1} to \eqref{e2-1} is
convergent. This means that systems \eqref{e1-1} and \eqref{e2-1} are analytically equivalent through a near identity change of variables.

Next we prove that system \eqref{e2-1} is analytically equivalent to system \eqref{e2}. Take the change of variables
\[
u=u,\quad v=v,\quad \rho=w+\varphi(u,v,w),
\]
for which system \eqref{e2-1} is transformed to \eqref{e2}. Then we have
\begin{eqnarray}\label{e2-2}
\frac{\partial \varphi }{\p w}Bw-\mf i\frac{\p \vp}{\p u} u+\mf i \frac{\p \vp}{\p v} v-B\vp&=&
Bh(u,v,w+\vp(u,v,w))\nonumber \\
&& -\frac{\p \vp}{\p w}wg+ug_1\frac{\p \vp}{\p u}-vg_2\frac{\p \vp }{\p v},
\end{eqnarray}
where $wg=(w_3g_3,\ldots,w_ng_n)^{tr}$, and we look $\vp$ as a column vector and $\frac{\p \vp}{\p w}$ is the Jacobian matrix of $\vp$ with respect to
$w$. The linear operator
\[
L=\frac{\partial  }{\p w}Bw-\mf i\frac{\p }{\p u} u+\mf i \frac{\p }{\p v} v-B,
\]
has the spectrum
\[
\{\langle k,\lambda\rangle-p\mf i+1\mf i-\lambda_j:\, k\in \mb Z_+^{n-2}, |k|=l,p,q\in \mb Z_+,  j=3,\ldots,n\},
 \]
in the linear space $\mc H^{l+p+q}$ which consists of $n-2$ dimensional vector valued homogeneous polynomials of degree $l$ in $w$ and of degrees $p$ and $q$
in $u$ and $v$, respectively.

Expanding $\vp, h, g_1, g_2$ and $g$ in the Taylor series, and equating the homogeneous terms in \eqref{e2-2} which have the same degree, we get from
induction and the assumption $(H)$ that equations \eqref{e2-2} have a formal series solution $\vp$ with its monomials all nonresonant.
Moreover, by the assumption that $A$ has its eigenvalues either all having positive real parts or all having negative real parts,  there exists a
number $\sigma>0$ such that if $\langle k,\lambda\rangle-p\mf i+1\mf i-\lambda_j\ne 0$ for $(p,l,k)\in\mb Z_+^n$  we have
\[
\|\langle k,\lambda\rangle-p\mf i+1\mf i-\lambda_j\|\ge \sigma.
\]
This shows that $\vp $ in the transformation does not contain small denominators. Then similar to the proof of the classical Poincar\'e--Dulac normal form theorem
we can prove that $\vp$ is convergent, see for instance \cite{Bi79,Zh08}, where similar proofs on convergence of $\vp$ were provided.
This proves statement $(b)$, and consequently the lemma.
\end{proof}

Next result shows the existence of analytic integrating factor on the center manifold provided the existence of analytic inverse Jacobian multiplier
of system \eqref{e1-1} in a neighborhood of the origin.

\begin{lemma}\label{l4}
Assume that system \eqref{e1-1} has an analytic inverse Jacobian multiplier of the form
\[
J(\xi,\eta,\zeta)=(\zeta_3-\phi_3(\xi,\eta,\zeta))\ldots (\zeta_n-\phi_n(\xi,\eta,\zeta))V(\xi,\eta,\zeta),
\]
with $\phi_j=O(|(\xi,\eta,\zeta)|^2)$ for $j\in \{3,\ldots,n\}$ and $V$ analytic, and  $V(0,0,0)\ne 0$.
Then
\begin{itemize}
\item[$(a)$] $\mc M=\bigcap\limits_{j=3}\limits^n\{\zeta_j=\phi_j(\xi,\eta,\zeta)\}$ is an invariant analytic center
manifold of $\widetilde{\mathcal X}$ in a neighborhood of the origin.
\item[$(b)$] $V|_{\mc M}$ is an analytic inverse integrating factor of $\widetilde{\mathcal X}|_{\mc M}$.
\end{itemize}
\end{lemma}

\begin{proof} $(a)$
By the expression of $J$ we get from $\widetilde{\mathcal X}(J)=J\mbox{div}\widetilde{\mathcal X}$ that
\begin{eqnarray*}
&&\sum\limits_{j=3}\limits^n \wt{ \mc X}(\zeta_j-\phi_j)(\zeta_3-\phi_3)\ldots \widehat{(\zeta_j-\phi_j)}\ldots (\zeta_n-\phi_n)V(\xi,\eta,\zeta)\\
&&\quad +(\zeta_3-\phi_3)\ldots  (\zeta_n-\phi_n)\wt{ \mc X}(V)=(\zeta_3-\phi_3)\ldots  (\zeta_n-\phi_n)V\mbox{div}\wt{ \mc X},
\end{eqnarray*}
where $\widehat{(\zeta_j-\phi_j)}$ denotes its absence in the product. Since $\zeta_3-\phi_3,\ldots,\zeta_n-\phi_n$ are relatively pairwise
coprime in the algebra of analytic functions which are defined in a neighborhood of the origin, so there exist analytic functions
\[
L_0(\xi,\eta,\zeta),\,\,\,\, L_3(\xi,\eta,\zeta),\,\,\,\, \ldots,\,\,\,\,  L_n(\xi,\eta,\zeta),
\]
such that
\begin{equation}\label{e1-52}
\begin{array}{rcl}
\widetilde{\mathcal X}(V(\xi,\eta,\zeta))&=& L_0(\xi,\eta,\zeta)V(\xi,\eta,\zeta),\\
\widetilde{\mathcal X}(\zeta_j-\phi_j(\xi,\eta,\zeta))&=& L_j(\xi,\eta,\zeta)(\zeta_j-\phi_j(\xi,\eta,\zeta)),
\end{array}
\end{equation}
for $j=3,\ldots, n$.
This shows that $\zeta_j=\phi_j(\xi,\eta,\zeta)$, $j=3,\ldots,n$, are invariant under the flow of $\widetilde{\mathcal X}$.

Applying the Implicit Function Theorem to the equations
\[
\zeta_j-\phi_j(\xi,\eta,\zeta)=0,\quad j=3,\ldots,n,
\]
we get a unique solution $\zeta=k(\xi,\eta)$, i.e.
\[
\zeta_j=k_j(\xi,\eta),\quad j=3,\ldots,n,
\]
in a neighborhood of the origin, which is analytic. Hence
\[
\mc M=\bigcap\limits_{j=3}\limits^n\{\zeta_j=k_j(\xi,\eta)\},
\]
in a neighborhood of the origin. Again the Implicit Function Theorem shows that
 $k_j(0,0)=0$ and $\p_\xi k_j(0,0)=\p_\eta k_j(0,0)=0$ for $j=3,\ldots,n$. These imply that $\mc M$ is an analytic center manifold of $\widetilde{\mathcal X}$
 in a neighborhood of the origin which is tangent to the $(\xi,\eta)$ plane.

\noindent $(b)$ Since
\[
\widetilde{\mathcal X}(\zeta_j-\phi_j(\xi,\eta,\zeta))=0 \quad \mbox{ on }  \mc M,
\]
we have
\begin{eqnarray*}
&& \widetilde F_j(\xi,\eta,k(\xi,\eta))-\wt F_1(\xi,\eta,k(\xi,\eta))\frac{\partial \phi_j}{\partial \xi}-\wt F_2(\xi,\eta,k(\xi,\eta)) \frac{\partial \phi_j}{\partial \eta}\\
&& \qquad\qquad\quad\qquad -\p_\zeta\phi_j \wt F(\xi,\eta,k(\xi,\eta))=0,\qquad j=3,\ldots,n.
\end{eqnarray*}
Here we have used the conventions $\p_\zeta\phi_j=(\p_{\zeta_3}\phi_j,\ldots,\p_{\zeta_n}\phi_j)$ and $\wt F=(\wt F_3,\ldots,\wt F_n)^{tr}$.
Write these equations in a unified vector form, we have
\begin{equation}\label{e1-50}
\left(E-\p_\zeta \phi\right)\wt F=\wt F_1\p_\xi \phi+\wt F_2\p_\eta \phi\qquad \mbox{ on }\,\, \mc M,
\end{equation}
where $\p_s\phi=(\p_s \phi_3,\ldots,\p_s\phi_n)^{tr}$, $s\in\{\xi,\eta\}$.

In addition, since
\[
k(\xi,\eta)=\phi(\xi,\eta,k(\xi,\eta)),
\]
we have
\begin{equation}\label{e1-51}
\left(E-\p_\zeta\phi\right)\p_\xi k=\p_\xi\phi,\quad \left(E-\p_\zeta\phi\right)\p_\eta k=\p_\eta\phi,
\end{equation}
where $\p_sk=(\p_s k_3,\ldots,\p_sk_n)^{tr}$, $s\in\{\xi,\eta\}$.

Set $C(\xi,\eta)=V(\xi,\eta,k(\xi,\eta))$. Some calculations show that
\begin{eqnarray}\label{e1-5}
\wt{\mathcal X}|_{\mc M}(C(\xi,\eta))&=&\wt F_1(\xi,\eta,k(\xi,\eta))\frac{\partial C}{\partial \xi}+\wt F_2(\xi,\eta,k(\xi,\eta))\frac{\partial C}{\partial \eta}\nonumber\\
&=&\wt F_1[w]\left(\frac{\partial V}{\partial \xi}+\frac{\partial V}{\partial \zeta_3}\frac{\partial k_3}{\partial \xi}+\ldots+\frac{\partial V}{\partial \zeta_n}\frac{\partial k_n}{\partial \xi}\right)\nonumber\\
&& +
\wt F_2[w]\left(\frac{\partial V}{\partial \eta}+\frac{\partial V}{\partial \zeta_3}\frac{\partial k_3}{\partial \eta}+\ldots+\frac{\partial V}{\partial \zeta_n}\frac{\partial k_n}{\partial \eta}\right)\\
&=&\wt F_1[w]\frac{\partial V}{\partial \xi} +  \wt F_1[w]\p_\zeta V\left(E-\p_\zeta \phi\right)^{-1}\p_\xi\phi\nonumber\\
&&\left.+\wt F_2[w]\frac{\partial V}{\partial \eta}+\wt F_2[w]\p_\zeta V\left(E-\p_\zeta \phi\right)^{-1}\p_\eta\phi\right|_{\mc M}\nonumber\\
&=&
\left.\wt F_1[w]\frac{\partial V}{\partial \xi} +\wt F_2[w]\frac{\partial V}{\partial \eta}+\p_\zeta V \wt F[w]\right|_{\mc M}\nonumber\\
&=&\wt{ \mathcal X}(V)|_{\mc M}=L_0V|_{\mc M}=L_0|_{\mc M} C,\nonumber
\end{eqnarray}
where $[w]=(\xi,\eta,k(\xi,\eta))$, and in  the third and fourth equalities we have used respectively \eqref{e1-51} and \eqref{e1-50}. Recall that $\p_\zeta V=
(\p_{\zeta_3}V,\ldots,\p_{\zeta_n}V)$.

Next we shall prove that $L_0|_{\mc M}=\mbox{div}(\wt {\mathcal X}|_{\mc M})$. From the definition of inverse Jacobian multipliers and \eqref{e1-52}, we get that
\begin{eqnarray*}
&& J\mbox{div}\wt{\mathcal X}=\wt{\mathcal X}(J)=(L_0+L_3+\ldots+L_n)J.
\end{eqnarray*}
This reduces to
\begin{equation}\label{e1-53}
L_0=\mbox{div}\wt{\mathcal X}-L_3-\ldots-L_n.
\end{equation}
Note that for $ j=3,\, \ldots,\, n$
\begin{eqnarray*}
 L_j(\zeta_j-\phi_j(\xi,\eta,\zeta))&=&\wt{\mathcal X} (\zeta_j-\phi_j(\xi,\eta,\zeta))\\
&=& \wt F_j-\wt F_1\p_\xi \phi_j-\wt F_2\p_\eta\phi_j-\p_\zeta\phi_j \wt F.
\end{eqnarray*}
Writing these equations in vector form gives
\begin{equation}\label{e1-55}
\mbox{diag}(L_3,\ldots,L_n)\left(\zeta-\phi(\xi,\eta,\zeta)\right)=\left(E-\p_\zeta\phi\right)\wt F-\p_\xi\phi \,\wt F_1-\p_\eta\phi\,\wt F_2.
\end{equation}
Recall that $\wt F,\,\p_\xi\phi,\, \p_\eta\phi$ are $n-2$ dimensional column vectors.

On the center manifold $\mc M$ we have
\[
\phi_j(\xi,\eta,\zeta)=\zeta_j, \qquad  j=3,\, \ldots,\, n.
\]
So from these we get that
\[
\p_\xi\p_{\zeta_s}\phi_j=\p_\eta\p_{\zeta_s}\phi_j=\p_{\zeta_s}\p_{\zeta_l}\phi_j=0 \,\,\, \mbox{ on } \, \mc M,\quad \mbox{for all } 3\le s,j,l\le n.
\]
Differentiating \eqref{e1-55} with respect to $\zeta$, together with these last equalities, yield
\begin{eqnarray*}
&&\mbox{diag}(L_3,\ldots,L_n)\left(E-\p_\zeta\phi\right)\\
&&=\left(E-\p_\zeta\phi\right)\p_\zeta\wt F-\p_\xi\phi \,\p_\zeta\wt F_1-\p_\eta\phi\,\p_\zeta\wt F_2\qquad \mbox{ on }\, \mc M.
\end{eqnarray*}
We note that $\p_\zeta\wt F$ is a matrix of order $n-2$, and $\p_\zeta\wt F_s$ for $s=1,2$ are $n-2$ dimensional horizontal vectors.
Rewrite this last equation in the following form
\begin{eqnarray}\label{et1}
\mbox{diag}(L_3,\ldots,L_n)&=&\left(E-\p_\zeta\phi\right)\p_\zeta\wt F\left(E-\p_\zeta\phi\right)^{-1}\nonumber\\
&& -\p_\xi\phi \,\p_\zeta\wt F_1\left(E-\p_\zeta\phi\right)^{-1}-\p_\eta\phi\,\p_\zeta\wt F_2\left(E-\p_\zeta\phi\right)^{-1}.
\end{eqnarray}
Since similar matrices have the same trace, we have
\begin{equation}\label{et2}
\mbox{trace}\left(\left(E-\p_\zeta\phi\right)\p_\zeta\wt F\left(E-\p_\zeta\phi\right)^{-1}\right)=\mbox{trace}(\p_\zeta\wt F)
=\sum\limits_{j=3}\limits^n\p_{\zeta_j}\wt F_j.
\end{equation}
Moreover some calculations show that
\begin{eqnarray}\label{et3}
\mbox{trace}\left(\p_\xi\phi \,\p_\zeta\wt F_1\left(E-\p_\zeta\phi\right)^{-1}\right)&=&
\mbox{trace}\left(\left(E-\p_\zeta\phi\right)^{-1}\p_\xi\phi\, \p_\zeta\wt F_1\right)\nonumber\\
&=&\p_\zeta\wt F_1\left(E-\p_\zeta\phi\right)^{-1}\p_\xi\phi,
\end{eqnarray}
and
\begin{eqnarray}\label{et4}
\mbox{trace}\left(\p_\eta\phi \,\p_\zeta\wt F_2\left(E-\p_\zeta\phi\right)^{-1}\right)&=&
\mbox{trace}\left(\left(E-\p_\zeta\phi\right)^{-1}\p_\eta\phi\, \p_\zeta\wt F_2\right)\nonumber\\
&=&\p_\zeta\wt F_2\left(E-\p_\zeta\phi\right)^{-1}\p_\eta\phi.
\end{eqnarray}
Combining \eqref{et1}, \eqref{et2}, \eqref{et3} and \eqref{et4} gives
\[
L_3+\ldots+L_n=\sum\limits_{j=3}\limits^n\p_{\zeta_j}\wt F_j
-\p_\zeta\wt F_1\left(E-\p_\zeta\phi\right)^{-1}\p_\xi\phi
-\p_\zeta\wt F_2\left(E-\p_\zeta\phi\right)^{-1}\p_\eta\phi.
\]
This together with \eqref{e1-53} show that
\begin{eqnarray}\label{et6}
L_0|_{\mc M}&=& \left.\p_\xi \wt F_1+\p_\eta\wt F_2+\p_\zeta\wt F_1\left(E-\p_\zeta\phi\right)^{-1}\p_\xi\phi\right.\nonumber \\
&& \left. +\p_\zeta\wt F_2\left(E-\p_\zeta\phi\right)^{-1}\p_\eta\phi\right|_{\mc M}\nonumber\\
&=& \left.\p_\xi \wt F_1+\p_\eta\wt F_2+\p_\zeta\wt F_1\,\p_\xi k
+\p_\zeta\wt F_2\,\p_\eta k\right|_{\mc M}  \\
&=&\p_\xi \wt F_1(\xi,\eta,k(\xi,\eta))+\p_\eta \wt F_2(\xi,\eta,k(\xi,\eta))=\mbox{div}\mathcal (\wt {\mc X}|_{\mc M}),\nonumber
\end{eqnarray}
where in the second equality we have used \eqref{e1-51}.

Now the equalities \eqref{e1-5} and \eqref{et6} verify that $C(\xi,\eta)$ is an analytic inverse integrating factor of the vector field
$\wt{\mathcal X} |_{\mc M}$.

We complete the proof of the lemma.
\end{proof}

\begin{remark} \label{l3}
Replacing analyticity by $C^\infty$ smoothness Lemma \ref{l4} holds, too.
\end{remark}

We now study the properties of $C^\infty$ inverse Jacobian multiplier restricted to center manifolds.

\begin{lemma}\label{l5}
Assume that system \eqref{e1} satisfies $(H)$ and has a $C^\infty$ inverse Jacobian multiplier, written in conjugate complex coordinates as
\[
J(\xi,\eta,\zeta)=\prod\limits_{j=3}\limits^n(\zeta_j-\psi_j(\xi,\eta,\zeta))V(\xi,\eta,\zeta),
\]
where $\psi_j=O(|(\xi,\eta,\zeta)|^2)$ and $V$ has no factor $\zeta_l-\psi_l(\xi,\eta,\zeta)$ for any $l\in\{3,\ldots,n\}$. Then the following statements hold.
\begin{itemize}
\item[$(a)$] $\mathcal M^*=\bigcap\limits_{j=3}\limits^n\{\zeta_j=\psi_j(\xi,\eta,\zeta)\}$ is a center manifold of system \eqref{e1} at the origin.
\item[$(b)$] For any smooth center manifold $\mc M$ of system \eqref{e1} at the origin, if ${\mathcal X}|_{\mathcal M}$ has the origin as a center, then $J|_{\mathcal M}=0$.
\end{itemize}
\end{lemma}

\begin{proof} $(a)$  As in the proof of Lemma \ref{l4} there exist $C^\infty$ smooth functions $L_3,\ldots,L_n$ such that
\[
\wt{\mc X}(\zeta_j-\psi_j(\xi,\eta,\zeta))=L_j(\xi,\eta,\zeta)(\zeta_j-\psi_j(\xi,\eta,\zeta), \quad j=3,\ldots,n,
\]
where $\wt{\mc X}$ is $\mc X$ written in the conjugate complex coordinates as those did in \eqref{e1-1}. Note that each surface $\zeta_j-\psi_j(\xi,\eta,\zeta)$ is
invariant under the flow of $\wt{\mc X}$. By the Implicit Function Theorem the equations
\[
\zeta_j-\psi_j(\xi,\eta,\zeta)=0,\quad j=3,\ldots,n,
\]
have a unique solution $\zeta=k(\xi,\eta)$, which is $C^\infty$. Representing $\zeta=k(\xi,\eta)$ in the cartesian coordinates gives
\[
z=h(x,y), \quad i.e. \,\,\, \, z_j=h_j(x,y), \quad j=3,\ldots,n.
\]
Clearly $\p_xh_j(0,0)=\p_yh_j(0,0)=0$ for $j=3,\ldots,n$. Then
\[
\mathcal M^*=\bigcap\limits_{j=3}\limits^n\{\zeta_j-k_j(\xi,\eta)\}=\bigcap\limits_{j=3}\limits^n\{z_j-h_j(x,y)\},
\]
is a center manifold of system \eqref{e1} at the origin.

\noindent $(b)$ Let $P_0=(x_0,y_0,z_0)$ be any point on $\mathcal M$ in a sufficiently small neighborhood of the origin,
and let $\varphi_t$ be the orbit of \eqref{e1} passing through $P_0$. Then we have
\[
\frac{dJ(\varphi_t)}{dt}=\mathcal X(J)|_{\varphi_t}=J\mbox{div}\mathcal X|_{\varphi_t}.
\]
This equation has the solution
\begin{equation}\label{ejf1}
J(\varphi_t(x_0,y_0,z_0))=J(x_0,y_0,z_0)\exp\left(\int_0^t\mbox{div}\mathcal X|_{\varphi_s}ds\right).
\end{equation}
By the assumption $\varphi_t$ is a periodic orbit. Denote its period by $T_0$. This last equation can be simplified to
\begin{equation}\label{e1-11}
J(x_0,y_0,z_0)=J(x_0,y_0,z_0)\exp\left(\int_0^{T_0}\mbox{div}\mathcal X|_{\varphi_s}ds\right).
\end{equation}
Restricted to the center manifold $\mathcal M$ system \eqref{e1} becomes
\begin{equation}\label{e1-8}
\dot x=-y+f_1(x,y,h(x,y)),\quad \dot y=x+f_2(x,y,h(x,y)).
\end{equation}
Written in polar coordinates $(x,y)=(r\cos\theta,r\sin\theta)$, we get from this two dimensional system
\[
\frac{d\theta}{1+O(r)}=dt.
\]
Integrating along the periodic orbit gives
\[
T_0=2\pi+O(r).
\]
So we have
\[
\int_0^{T_0}\mbox{div}\mathcal X|_{\varphi_s}ds=\int_0^{T_0}(\Lambda+O(|P_0|))ds=2\pi\Lambda+O(|P_0|).
\]
This together with \eqref{e1-11} yields that in a sufficiently small neighborhood of the origin
\[
J(x_0,y_0,z_0)=0.
\]
By the arbitrariness of $P_0\in\mathcal M$ we get that $J|_{\mathcal M}\equiv 0$. This proves statement $(b)$.

We complete the proof of the lemma.
\end{proof}

Having the above preparations we can prove Theorems \ref{t1} and \ref{t2}.

\subsection{Proof of Theorem \ref{t1}}  $(a)$
{\it Sufficiency}. If the matrix $A$ has conjugate complex eigenvalues, we write  system \eqref{e1} in \eqref{e1-1}. By Lemma \ref{l4} and its proof we get that
system \eqref{e1-1} has an analytic center manifold $\mathcal M=\bigcap\limits_{j=3}\limits^n\{\zeta_j=k_j(\xi,\eta)\}$. Again by Lemma \ref{l4},
system \eqref{e1-1} restricted to $\mathcal M$, i.e. \eqref{e1-8}, has an analytic inverse integrating factor $C(\xi,\eta)=\wt V(\xi,\eta,k(\xi,\eta))$, where
$\wt V$ is $V(x,y,z)$ written in $(\xi,\eta,\zeta)$.

We note that either $\zeta_j=z_j$ is a real coordinate or $\zeta_j=z_j+\mf i z_{j+1}$ and $\zeta_{j+1}=z_j-\mf i z_{j+1}$ for some $j$ are conjugate complex
coordinates.
In the latter write $k_j(\xi,\eta)=h_j(x,y)+\mf i h_{j+1}(x,y)$, we have $z_j=h_j(x,y)$. In the former write $h_j=k_j$. Then we have
\[
\mathcal M=\bigcap\limits_{j=3}\limits^n\{z_j=k_j(x,y)\}.
\]

Since $C(0,0)=V(0,0,0)\ne 0$, integrating the one--form
\[
\frac{x+f_2(x,y,h(x,y))}{V(x,y,h(x,y))}dx+\frac{y-f_1(x,y,h(x,y))}{V(x,y,h(x,y))}dy,
\]
provides an analytic first integral $H(x,y)$ of \eqref{e1-8} and it has the form  $H(x,y)=(x^2+y^2)/C(0,0)+\mbox{higher order term}$.
So we get from the Poincar\'e center theorem that the vector field $\mathcal X$ has the origin as a center on the center manifold $\mathcal M$.

The vector field $\mathcal X$ restricted to the center manifold $\mc M$ has the origin as a center and it has an analytic first integral.
These facts together with Theorems 6.3 and 7.1 of Sijbrand \cite{Si85} show that the center manifold at the origin is unique and analytic.
So we have $\mc M^c=\mc M$. Hence system \eqref{e1} restricted to $\mc M^c$ has the origin as a center.

\noindent {\it Necessity.}  First we write system \eqref{e1} in \eqref{e1-1} with the conjugate complex coordinates. Lemmas \ref{l1} and \ref{l2} show that system \eqref{e1-1} is analytically equivalent to its distinguished normal form, i.e.  system \eqref{e2}, through a distinguished normalization.

For the analytic differential system \eqref{e2} we have $g_1=g_2$ by the proof of Lemma \ref{l2}. We can check easily that $\widetilde  J=w_3\ldots w_n$ is an inverse Jacobian multiplier of system \eqref{e2} and is clearly analytic. Hence using the near identity analytic transformation from \eqref{e1-1} to \eqref{e2} we get that system \eqref{e1-1} has an analytic inverse Jacobian multiplier
\[
J^*=(\zeta_3-\phi_3(\xi,\eta,\zeta))\ldots(\zeta_n-\phi_n(\xi,\eta,\zeta))/D(\xi,\eta,\zeta),
\]
where $D(\xi,\eta,\zeta)$ is the determinant of the Jacobian matrix of the transformation from \eqref{e1-1} to \eqref{e2}, and satisfies $D(0,0,0)=1$.

Going back to the $(x,y,z)$ coordinates we get that system \eqref{e1} has an analytic inverse Jacobian multiplier of the form \eqref{*1}.

\noindent $(b)$ The analyticity and uniqueness of the center manifolds were proved in the sufficient part of statement $(a)$.
$\mathcal M^c\subset J^{-1}(0)$ follows from Lemma \ref{l5}\,$(b)$ and the first assertion.

We complete the proof of the theorem.  \qquad $\Box$

\subsection{Proof of Theorem \ref{t2}} $(a)$ Under the assumption of the theorem, we get from Lemma \ref{l2}\,$(a)$ that system \eqref{e1-1} is locally $C^\infty$
equivalent to its Poincar\'e--Dulac normal form \eqref{e2} with $g_1\ne g_2$. Direct calculations show that system \eqref{e2} has the $C^\infty$ inverse
Jacobian multiplier
\[
\widetilde J=w_3\ldots w_n uv (g_2(uv)-g_1(uv)),
\]
where $g_1(s),g_2(s)$ are $C^\infty$ functions and $g_2-g_1$ is non--flat at $s=0$. This shows that $\widetilde J=w_3\ldots w_n (uv)^lh(uv)$
with $l\ge 2$ and $h(0)\ne 0$. Without loss of generality we can assume $h(0)=1$. By the inverse transformations from \eqref{e1-1} to \eqref{e2}
we get that system \eqref{e2} has a $C^\infty$ inverse Jacobian multiplier of the form
\[
J(\xi,\eta,\zeta)=\prod\limits_{j=3}\limits^n(\zeta_j-\phi_j(\xi,\eta,\zeta))((\xi-\phi_1)(\eta-\phi_2))^lh((\xi-\phi_1)(\eta-\phi_2))/D(\xi,\eta,\zeta),
\]
where the $C^\infty$ smothness follows from the facts that $\widetilde J$ and the near identity transformation from \eqref{e1-1} to \eqref{e2} are
both $C^\infty$  smooth,
$D(\xi,\eta,\zeta)$ is the determinant of the Jacobian matrix of the transformation satisfying $D(0,0,0)=1$.

Note that $\phi_1$ and $\phi_2$ are conjugate. And for $j=3,\ldots,n$, either $\phi_j$ is real if $\zeta_j$ is real, or $\phi_j$ and $\phi_k$ are conjugate if some $\zeta_j$
and $\zeta_k$ are conjugate. So written the conjugate complex coordinates $(\xi,\eta,\zeta)$ (if exist) in the real ones $(x,y,z)$ we get that
system \eqref{e1} has the inverse Jacobian multiplier in the prescribed form \eqref{*2}.

\noindent $(b)$  The proof follows from statement $(a)$ and Lemma \ref{l5}.

We complete the proof of the theorem. \qquad \qquad $\Box$

\medskip

\section{Proof of Theorem \ref{t3}}\label{s3}

\subsection{Preparation to the proof}\label{s3.1} Under the assumption of Theorem \ref{t3} we get from Theorem \ref{t2} that system \eqref{e1} has a $C^\infty$ inverse Jacobian multiplier of the form
 \begin{eqnarray*}
    J(x,y,z)&=&\prod\limits_{j=0}\limits^{m-1}\left[(z_{3+2j}-\psi_{3+2j}(x,y,z))^2+(z_{3+2j+1}-\psi_{3+2j+1}(x,y,z))^2\right]\\
     && \qquad \times\prod\limits_{s=3+2m}\limits^n (z_s-\psi_s(x,y,z))(x^2+y^2)^lV(x,y,z),
    \end{eqnarray*}
where $V(0,0,0)=1$. Moreover, it follows from the proof of Lemmas \ref{l4} and \ref{l5} that system \eqref{e1} has a $C^\infty$
center manifold $\mc M^c$ at the origin, which is defined by the intersection of the invariant surfaces
\begin{equation}\label{e31-1}
z_j=\psi_j(x,y,z), \qquad j=3,\ldots,n.
\end{equation}
Furthermore the center manifold can be represented as $\mc M^c=\bigcap\limits_{j=3}\limits^n \{z_j=h_j(x,y)\}$, where $z=h(x,y)$ is the unique solution of
\eqref{e31-1} defined in a neighborhood of the origin, which is obtained from the Implicit Function Theorem.
 Recall that $z=(z_3,\ldots,z_n)$ and $h=(h_3,\ldots,h_n)$.

If $m>0$, set for $j=0,\ldots, m-1$
\begin{eqnarray*}
&&\zeta_{3+2j}=z_{3+2j}+\mf i z_{3+2j+1},\qquad\quad \psi_{3+2j}^*(x,y,\zeta)=\psi_{3+2j}+\mf i \psi_{3+2j+1},\\
&&\zeta_{3+2j+1}=z_{3+2j}-\mf i z_{3+2j+1},\qquad \psi_{3+2j+1}^*(x,y,\zeta)=\psi_{3+2j}-\mf i \psi_{3+2j+1}.
\end{eqnarray*}
Note that the determinant of the Jacobian matrix of the transformation from $(x,y,z)$ to $(x,y,\zeta)$ is a nonzero constant.
If $\psi_j(x,y,z)\ne 0$, we take the
change of variables
\begin{equation}\label{e3.1}
(u,v,w)=\Phi(x,y,z)=(x,  y,  \zeta-\psi^*(x,y,\zeta)),
\end{equation}
where $\zeta=(\zeta_3,\ldots,\zeta_{2m+2},z_{2m+3},\ldots,z_n)$ and
$\psi^*=(\psi_3^*,\ldots,\psi_{2m+2}^*,\psi_{2m+3},\ldots,\psi_n)$.
Then system \eqref{e1} is transformed to
\begin{equation}\label{e3.2}
\begin{array}{l}
\dot u=-v+g_1(u,v,w),\\
\dot v=\,\,\, \, u+g_2(u,v,w),\\
\dot w_j=w_j(\lambda_j+g_j(u,v,w)), \qquad j=3,\ldots,n
\end{array}
\end{equation}
where $g_1,g_2=O(|(u,v,w)|^2)$ and $g_j=O(|(u,v,w)|)$, $j=3,\ldots,n$. Correspondingly system \eqref{e3.2} has the center manifold $w=0$. Moreover system \eqref{e3.2} has the associated inverse Jacobian multiplier
\[
\frac{J\circ \Phi^{-1}(u,v,w)}{D\Phi^{-1}(u,v,w)}=w_3\ldots w_n (u^2+v^2)^l\wt V(u,v,w)
\]
where $\wt V(0,0,0)\ne 0$, and $D\Phi^{-1}$ is the determinant of the Jacobian matrix of $\Phi^{-1}$ with respect to
its variables and $D\Phi(0,0,0)=1$.

Since systems \eqref{e1} and \eqref{e3.2} are $C^\infty$ equivalent in a neighborhood of the origin and the corresponding inverse Jacobian multipliers have the same forms, so in what follows we assume without loss of generality that system \eqref{e1} has the center manifold $z=0$ and the coordinate hyperplane $z_j=0$ is invariant for $j=3,\ldots,n$.

Taking the cylindrical coordinate changes
\[
x=r\cos \theta,\quad y=r\sin\theta, \quad z=r s,
\]
with $r\ge 0$, system \eqref{e5} is transformed to
\begin{equation}\label{e3.3}
\dot \theta=1+\Theta(\theta,r,s,\ve), \quad
\dot r=R(\theta,r,s,\ve) \quad
\dot s=As+S(\theta,r,s,\ve),
\end{equation}
where
\begin{eqnarray*}
\Theta(\theta,r,s,\ve)&=&\frac{\cos\theta g_2(r\cos\theta,r\sin\theta,rs,\ve)-\sin\theta g_1(r\cos\theta,r\sin\theta,rs,\ve)}{r},\\
R(\theta,r,s,\ve)&=&\cos\theta g_1(r\cos\theta,r\sin\theta,rs,\ve)+\sin\theta g_2(r\cos\theta,r\sin\theta,rs,\ve),\\
S(\theta,r,s,\ve)&=&\frac{g(r\cos\theta,r\sin\theta,rs,\ve)-sR(\theta,r,s,\ve)}{r},
\end{eqnarray*}
where $g=(g_3,\ldots,g_n)^{tr}$ with $g_3,\ldots,g_n$ given in \eqref{e3.2}.
Notice that
\begin{eqnarray*}
&& R(\theta,0,s,\ve)=0, \,\,\,\,\qquad  R(\theta,r,s,0) =O(r^2),\\
&& \Theta(\theta,r,s,0)=O(r), \,\, \quad S(\theta,r,s,0)=O(r).
\end{eqnarray*}
Corresponding to the inverse Jacobian multiplier $J(x,y,z)$ of system \eqref{e1}, system \eqref{e3.3} with $\ve=0$ has the inverse Jacobian multiplier
\begin{equation}\label{ejf2}
J(r\cos\theta,r\sin\theta,rs)/r^{n-1}=s_3\ldots s_nr^{2l-1} k(\theta,r,s),
\end{equation}
with $k(\theta,0,0)=$\,constant\,$\ne 0$.

For $|\ve|\ll 1$ and $|r|$ suitably small, we always have $\dot\theta > 0$. So system \eqref{e3.3} can be equivalently written in
\begin{equation}\label{e3.4}
\begin{array}{l}
\dfrac{dr}{d\theta}=\dfrac{R(\theta,r,s,\ve)}{1+\Theta(\theta,r,s,\ve)}=:p(\theta,r,s,\ve),\\
\dfrac{ds}{d\theta}=\dfrac{As+S(\theta,r,s,\ve)}{1+\Theta(\theta,r,s,\ve)}=:As+q(\theta,r,s,\ve).
\end{array}
\end{equation}
Furthermore, we have
\begin{equation}\label{e3.5}
p(\theta,0,s,\ve)=0, \quad p(\theta,r,s,0)=O(r^2), \quad q(\theta,r,s,0)=O(r).
\end{equation}
And $q=(q_3,\ldots,q_n)^{tr}$ with $q_j$ having the factor $s_j$ when $\ve=0$ for $j=3,\ldots,n$.

Associated to system \eqref{e3.4} we have a  vector field
\[
\mc Y_\ve=\partial_\theta+p(\theta,r,s,\ve)\partial_r+\langle As+q(\theta,r,s,\ve), \partial_s\rangle,
\]
where $\partial_s=(\partial_{s_3},\ldots,\partial_{s_n})$.
Related to the inverse Jacobian multiplier \eqref{ejf2} of system \eqref{e3.3}, the vector field $\mc Y_0$ has the inverse Jacobian multiplier
\begin{equation}\label{e3.6}
J_c(\theta,r,s)=\frac{J(r\cos\theta,r\sin\theta,rs)}{r^{n-1}(1+\Theta(\theta,r,s,0)}=s_3\ldots s_nr^{2l-1}K(\theta,r,s),
\end{equation}
where $K=1+O(r)$.

Clear $p,q$ are periodic in $\theta$ with period $2\pi$, and they are well defined on the cylinder
$\mc C=\{(\theta,r,s,\ve)\in\mathbb R/(2\pi\mathbb R)\times \mathbb R^{n-1}\times \mb R^m:\, |r|,|\ve|\ll 1\}$. Furthermore we note that each periodic orbit
of system \eqref{e5} corresponds to a unique periodic orbit of system \eqref{e3.4} on $\mc C$. So, to study the periodic orbits of system \eqref{e5}
is equivalent to study the periodic orbits of system \eqref{e3.4}.

Denote by $\psi_\theta(r_0,s_0,\ve)$ the solution of system \eqref{e3.4} with the initial point $\psi_0(r_0,s_0,\ve)=(r_0,s_0)\in \mc C$. We have
\[
\psi_\theta(r_0,s_0,\ve)=(r_\theta(r_0,s_0,\ve),s_\theta(r_0,s_0,\ve)).
\]
On the cylinder $\mc C$, $\theta=2\pi$
coincides with $\theta=0$. We define the Poincar\'e map on the transversal section $\theta=0$ of the flow of \eqref{e3.4} by
\[
\mc P(r_0,s_0;\ve)=\psi_{2\pi}(r_0,s_0,\ve).
\]
Since system \eqref{e3.4} is analytic, and so is the Poincar\'e map $\mc P$. Set
\[
\mc P(r_0,s_0,\ve)=(\mc P_r(r_0,s_0,\ve),\mc P_s(r_0,s_0,\ve)),
\]
with
\[
\mc P_r(r_0,s_0,\ve)=r_{2\pi}(r_0,s_0,\ve) \quad \mbox{ and } \quad
\mc P_s(r_0,s_0,\ve)=s_{2\pi}(r_0,s_0,\ve).
\]
Then
\begin{eqnarray*}
\mc P_r(r_0,s_0,\ve)&=& r_0+\int_0^{2\pi} p(v, r_v(r_0,s_0,\ve),s_v(r_0,s_0,\ve),\ve),s,\ve)dv,\\
\mc P_s(r_0,s_0,\ve))&=& e^{A2\pi}\left(Es_0+\int_0^{2\pi} e^{-Av}q(v, r_v(r_0,s_0,\ve),s_v(r_0,s_0,\ve),\ve),\ve)dv\right),
\end{eqnarray*}
where $E$ is the unit matrix of order $n-2$.

Define the displacement function by
\[
\mc D(r_0,s_0,\ve)=\mc P(r_0,s_0,\ve)-(r_0,s_0).
\]
Then the periodic orbit of system \eqref{e3.4} is uniquely determined by the zero of the displacement function $\mc D$. Set
\[
\mc D_r=\mc P_r-r_0,\quad
\mc D_s=\mc P_s-s_0.
\]
Then $\mc D=(\mc D_r,\,\mc D_s)$.

In order to study the zeros of $\mc D(r_0,s_0,\ve)$ on $(r_0, s_0)$ for any fixed $\ve$ sufficiently small, we will solve $\mc D_s(r_0,s_0,\ve)=0$ in $s_0$ as a
function of $(r_0,\ve)$ in a small neighborhood of $(r_0,\ve)=(0,0)$. In fact, by \eqref{e3.5} we get easily that
\[
\mc D_s(0,0,0)=0, \qquad \frac{\partial \mc D_s}{\partial s}(0,0,0)=e^{2\pi A}-E.
\]
These together with the assumption on $A$ show that the matrix $e^{2\pi A}-E$ is invertible. So the Implicit Function Theorem yields that $D_s(r_0,s_0,\ve)=0$ has a unique solution
$s_0=s^*(r_0,\ve)$ in a neighborhood of $(r_0,\ve)=(0,0)$, which is analytic.  Substituting $s^*$ into $\mc D_r$ gives
\[
d(r_0,\ve):=\mc D_r(r_0,s^*(r_0,\ve),\ve).
\]
Note that $d(r_0,\ve)$ is analytic. Thus the number of periodic orbits of system \eqref{e3.4} is equal to the number of positive roots $r_0$ of $d(r_0,\ve)=0$.

Having the above preparation we can prove Theorem \ref{t3}.

\subsection{Proof of Theorem \ref{t3}}\label{s3.2}

As we discussed in Subsection \ref{s3.1}, for proving Theorem \ref{t3} we only need to study the number of zeros of $d(r_0,\ve)$ in $r_0$.

From the expression of the inverse Jacobian multiplier $J_c$ it follows that $J_c$ is periodic in $\theta$ with period $2\pi$.
The inverse Jacobian multiplier $J_c$ and the Poincar\'e map $\mc P(r_0,s_0,0)$ of system \eqref{e3.4} with $\ve =0$ has the relation
\begin{equation}\label{e3.7}
J_c(0,\mc P(r_0,s_0,0))=J_c(0,r_0,s_0)D\mc P(r_0,s_0,0),
\end{equation}
where $D\mc P$ denotes the determinant of the Jacobian matrix of $\mc P$ with respect to $(r_0,s_0)$. For a proof, see \cite{BG13}. 
Here for completeness we provide a proof. From \eqref{ejf1} we have
\begin{equation}\label{ejf3}
J_c(0,\vp_\theta(r_0,s_0,0))=J_c(0,\vp_0(r_0,s_0,0))\exp\left(\int_0^\theta\mbox{div}\mathcal Y_0\circ \vp_s(r_0,s_0,0)ds\right),
\end{equation}
where $\vp_\theta(r_0,s_0,\ve)$ is the flow of the vector field $\mathcal Y_\ve$ or of system \eqref{e3.4} satisfying $\vp_0(r_0,s_0,\ve)=(r_0,s_0)$. Restricted \eqref{ejf3} to $\theta=2\pi$ and by the definition of the Poincar\'e map, we have
  \begin{equation}\label{ejf4}
J_c(0,\mathcal P(r_0,s_0,0))=J_c(0,r_0,s_0))\exp\left(\int_0^{2\pi}\mbox{div}\mathcal Y_0\circ \vp_s(r_0,s_0,0)ds\right),
\end{equation}
  Since the Jacobian matrix $\dfrac{\partial \vp_\theta(r_0,s_0,\ve)}{\partial (r_0,s_0)}$ satisfies the variational equations of system \eqref{e3.4} along the solution $(r,s)=\vp_\theta(r_0,s_0,\ve)$,
\[
\frac{d Z}{d\theta}=\frac{\partial(p,As+q)}{\partial (r,s)}\circ\vp(r_0,s_0,\ve)Z.
\]
By the Liouvellian formula we have
\[
\det\dfrac{\partial \vp_\theta(r_0,s_0,\ve)}{\partial (r_0,s_0)}=\det\dfrac{\partial \vp_0(r_0,s_0,\ve)}{\partial (r_0,s_0)}\exp\left(\int_0^\theta\mbox{div}\mathcal Y_\ve\circ \vp_s(r_0,s_0,\ve)ds\right).
\]
Taking $\ve=0$ and $\theta=2\pi$, this last equation can be written in
\[
\det\dfrac{\partial \vp_{2\pi}(r_0,s_0,0)}{\partial (r_0,s_0)}=\exp\left(\int_0^{2\pi}\mbox{div}\mathcal Y_0\circ \vp_s(r_0,s_0,0)ds\right).
\]
This together with \eqref{ejf4} verify \eqref{e3.7}.

Writing \eqref{e3.7} in components and using \eqref{e3.6}, we have
\begin{equation}\label{e3.8}
\mc P_{s3}\ldots \mc P_{sn} \mc P_r^{2l-1}K(0,\mc P_r,\mc P_s) =s_{03}\ldots s_{0n} r_0^{2l-1}K(0,r_0,s_0)D\mc P(r_0,s_0,0),
\end{equation}
where $\mc P_s=(\mc P_{s3},\ldots,\mc P_{sn})$.

Since the hyperplane $s_j=0$ is invariant under the flow of \eqref{e3.4} with $\ve=0$ for $j=3,\ldots,n$, we get that
\begin{equation}\label{e3.9}
\mc P_s(r_0,s_0,0)=(s_{03}\mc P_{s3}^*(r_0,s_0),\ldots,s_{0n}\mc P_{sn}^*(r_0,s_0))=:\langle s_0,\mc P_s^*(r_0,s_0)\rangle,
\end{equation}
where $s_0=(s_{03},\ldots,s_{0n})$. These together with \eqref{e3.8} show that
\begin{equation}\label{e3.10}
P_{s3}^*\ldots \mc P_{sn}^*\mc P_r^{2l-1}K(0,\mc P_r,\mc P_s)= r_0^{2l-1}K(0,r_0,s_0)D\mc P(r_0,s_0,0).
\end{equation}
Direct calculations show that $D\mc P(r_0,s_0,0)|_{s_0=0}= \mc P_{s3}^*\ldots \mc P_{sn}^* \partial_r\mc P_r|_{s_0=0}$.  So \eqref{e3.10} is simplified to
\begin{equation}\label{e3.11}
 \mc P_r^{2l-1}K(0,\mc P_r,0)= r_0^{2l-1}K(0,r_0,0)\partial_r\mc P_r\qquad \mbox{ for } s_0=0.
\end{equation}

From \eqref{e3.9} it follows that the solution $s^*(r_0,0)$  of $\mc D_s(r_0,s_0,\ve)=0$ with $\ve=0$ satisfies $s^*(r_0,0)\equiv 0$. So we have $
d(r_0,0)=\mc D(r_0,0,0)$. Expanding $d(r_0,0)$ in the Taylor series gives
\[
d(r_0,0)=\delta_kr_0^k+O(r_0^{k+1}),
\]
with $\delta_k\ne 0$ a constant. Then
\[
\mc P_r(r_0,0,0)=r_0+\delta_kr_0^k+O(r_0^{k+1}).
\]
Consequently we have $K(0,\mc P_r,0)=K(0,r_0,0)+O(r_0^k)$.
Substituting these expressions in \eqref{e3.11}, with some simple calculations, gives
\[
K(0,r_0,0)\left[(2l-1)\delta_kr_0^{2l-2+k}+O(r^{2l-1+k})\right]+O(r_0^{2l-1+k})=K(0,r_0,0)k\delta_kr_0^{2l-2+k}.
\]
Since $K(0,0,0)=1$, equating the coefficients of $r_0^{2l-2+k}$  in the last equation we get
\[
k=2l-1.
\]
Note that $l\ge 2$ by Theorem \ref{t2}, it follows that $k\ge 3$.

From the expression of $d(r_0,0)$ and the Weierstrass Preparation Theorem we get that $d(r_0,\ve)$ has at most $2l-1$ zeroes. Since system \eqref{e3.3}
is invariant under the symmetric change of variables $(\theta,r,s)\rightarrow (\theta+\pi,-r,-s)$, and $r_0=0$ is always a solution of $d(r_0,\ve)=0$,
these verify that $d(r_0,\ve)=0$ has at most $l-1$  positive roots.

We note that the $2\pi$ periodic solutions of \eqref{e3.4} one to one correspond to periodic orbits of \eqref{e5} in a neighborhood of the origin.
While each $2\pi$ periodic solution of \eqref{e3.4} in a neighborhood of the origin is uniquely determined by a positive zero of $d(r_0,\ve)$. So system \eqref{e5}
has at most $l-1$ small amplitude limit cycles which are bifurcated from the Hopf on the two dimensional center manifold.

Finally we provide an example showing that there exist systems of form \eqref{e5} which do have $l-1$ limit cycles under sufficient small perturbation. Consider
a special perturbation to system \eqref{e3.2}
\begin{equation}\label{e11}
\begin{array}{l}
\dot u=-v+g_1(u,v,w)+uh(u,v,\ve),\\
\dot v=\,\,\,\, u+g_2(u,v,w)+vh(u,v,\ve),\\
\dot w_j=w_j(\lambda_j+g_j(u,v))+w_jh(u,v,\ve), \qquad j=3,\ldots,n
\end{array}
\end{equation}
with $h(u,v,\ve)=\sum\limits_{s=1}\limits^{l-1}\ve^{l-s}a_s(u^2+v^2)^s$ and $\ve$ a single paramter. Recall that if $\lambda_j$ is complex with nonvanishing imaginary part, it must have a
conjugate one, saying $\lambda_{j+1}$, then the variables $w_j$ and $w_{j+1}$ are conjugate complex ones. Write system \eqref{e11} in cylindrical coordinates $(\theta,r,s)$, we get a system as in the form \eqref{e3.3} with $\Theta(\theta,r,s,\ve)$ and $S(\theta,r,s,\ve)$ independent of $\ve$, and $R(\theta,r,s,\ve)=R(\theta,r,s,0)+\sum\limits_{s=1}\limits^{l-1}\ve^{l-s}a_sr^{2s+1}$. Then similar to \cite{BGM13,GGG11} we get that for $|\ve|\ll 1$ and suitable choices of $a_1,\ldots,a_{l-1}$ system \eqref{e11} can have $l-1$ small amplitude limit cycles in a neighborhood of the origin.

We complete the proof of the theorem. \qquad \qquad $\Box$

\medskip

\noindent{\bf Acknowledgements.} The author is partially supported by
NNSF of China grant 11271252,  RFDP of Higher Education of China grant 20110073110054, and FP7-PEOPLE-2012-IRSES-316338 of Europe.

\bigskip
\bigskip
\bibliographystyle{amsplain}

\end{document}